\documentclass[12pt]{amsart}
\usepackage{geometry}
\usepackage{graphicx}	
\usepackage{amsthm,amsmath,amssymb}
\usepackage[all]{xy}                   

\newtheorem{thm}{Theorem}[section]
\newtheorem{cor}[thm]{Corollary}

\newtheorem{lem}[thm]{Lemma}
\newtheorem{ex}[thm]{Example}

\newcommand{\A}{\mathcal{A}}
\newcommand{\Z}{\mathbb{Z}}
\newcommand{\N}{\mathbb{N}}
\newcommand{\pos}{\mathbb{P}}
\newcommand{\Lin}{\mathcal{L}}
\newcommand{\collin}{\mathcal{CL}}
\newcommand{\symn}{\mathfrak{S}_{n}}

\newcommand{\colpermrn}{G_{r,n}}
\newcommand{\zigipi}{Z(I,\pi)}
\newcommand{\chainipi}{C(I,\pi)}
\newcommand{\dsum}{\displaystyle\sum}

\newcommand\mchoose[2]{\left(\!\! \binom{#1}{#2}\!\!\right)}
\newcommand\rversion[1]{{#1}_{(r)}}
\newcommand\shuffle[1]{\shuf(#1)}
\newcommand{\anchoredcolpermrn}{\colpermrn^0}
\newcommand{\rfoldx}{X_{(r)}}
\newcommand{\pless}{\prec}
\newcommand{\pgreat}{\succ}

\DeclareMathOperator{\des}{des}
\DeclareMathOperator{\Des}{Des}
\DeclareMathOperator{\intdes}{intdes}
\DeclareMathOperator{\intDes}{intDes}
\DeclareMathOperator{\shuf}{Sh}

\title{The Colored Eulerian Descent Algebra}
\author{Matthew Moynihan}
\address{Department of Mathematics and Computer Science \\
The College of Wooster \\
Wooster, OH 44691}
\email{mmoynihan@wooster.edu}
\date{\today}

\begin{document}

\begin{abstract}
Using a new colored analogue of $P$-partitions, we prove the existence of a colored Eulerian descent algebra which is a subalgebra of the Mantaci-Reutenauer algebra.
This algebra has a basis consisting of formal sums of colored permutations with the same number of descents (using Steingr{\'{\i}}msson's definition of the descent set of a colored permutation).
The colored Eulerian descent algebra extends familiar Eulerian descent algebras from the symmetric group algebra and the hyperoctahedral group algebra to colored permutation group algebras.
We also describe a set of orthogonal idempotents that spans the colored Eulerian descent algebra and includes, as a special case, the familiar Eulerian idempotents in the group algebra of the symmetric group.
\end{abstract}

\maketitle


In \cite{Solomon1976}, Solomon proved the existence of a subalgebra of the group algebra of the symmetric group whose basis elements are formal sums of permutations with the same descent set (the set of all $i$ such that $\pi(i) > \pi(i+1)$).
Solomon also proved that an analogous algebra can be defined for all finite Coxeter groups using the length function associated to a specified generating set.
Loday later proved in \cite{Loday1989} that Solomon's descent algebra contains a subalgebra whose basis elements are formal sums of permutations with the same number of descents.
This algebra is often called the Eulerian subalgebra.
Fran{\c{c}}ois Bergeron and Nantel Bergeron extended these results to the hyperoctahedral group (see \cite{BergeronBergeron1992I, BergeronBergeron1992, Bergeron1992II}).

In \cite{MantaciReutenauer1995}, Mantaci and Reutenauer proved the existence of an algebra whose basis elements are formal sums of colored permutations with the same associated colored compositions.
Colored permutations can be thought of as permutations from the symmetric group together with a choice of a ``color'' from a finite cyclic group for each letter in a given permutation.
Colored compositions describe both the colors of the letters of a colored permutation and the descent set within each monochromatic run.
The Mantaci-Reutenauer algebra has been well studied and fills much the same role as Solomon's descent algebra does in the theory of the group algebra of the symmetric group (see \cite{BaumannHohlweg2008, Poirier1998}).

While colored permutations do not usually form a Coxeter group, we show that there still exists an Eulerian subalgebra of the Mantaci-Reutenauer algebra when descents are defined using Steingr{\'{\i}}msson's definition of the descent set of a colored permutation from \cite{Steingrimsson1994}.
Our proof is a modified version of Petersen's proof in \cite{Petersen2005}, which was in turn inspired by Gessel's work on multipartite $P$-partitions in \cite{Gessel1984}.
We also describe a set of colored Eulerian idempotents which spans this colored Eulerian descent algebra and reduces to the set of well-known Eulerian idempotents in the group algebra of the symmetric group when considering permutations of a single color.

This paper is organized as follows.
In Section~\ref{sec: colored permutation groups}, we formally define colored permutations and the descent set of a colored permutation.
In Section~\ref{sec: colored posets} and Section~\ref{sec: colored p-partitions}, we define colored posets and colored $P$-partitions.
Our definitions are different from those of Hsiao and Petersen in \cite{HsiaoPetersen2010}.
They define colored posets to be posets with colored labels and obtain results about the Mantaci-Reutenauer algebra.
Our approach differs in the way we define the set of linear extensions of a given colored poset and allows us to count colored permutations by descent number.
In Section~\ref{sec: colored descent algebra}, we prove that the descent number induces an algebra which is a subalgebra of the Mantaci-Reutenauer algebra.
Finally, in Section~\ref{sec: colored eulerian idempotents}, we describe a set of colored Eulerian idempotents.


\section{Colored Permutation Groups}
\label{sec: colored permutation groups}

Let $[n]$ denote the set of integers $\{1,2,\ldots,n\}$ and let $[0,n]$ denote the set of integers $\{0,1,2,\ldots,n\}$.
Both sets are totally ordered using the natural ordering of the integers.
We denote by $\mathfrak{S}_n$ the group of permutations of $[n]$.
For every totally ordered set $X$,  we denote by $\rfoldx$ the set $[0,r-1] \times X$ ordered lexicographically.
Thus $(i,x) < (j,y)$ in $\rfoldx$ if either $i<j$ as integers or if both $i=j$ and $x < y$ in $X$.
For notational convenience, we often denote the subset $\{k\} \times X$ by $X_k$ and write $x_k$ in place of $(k,x)$.
For example, $[n]_{(r)} = \{i_j \, : \, 1 \leq i \leq n, \, 0 \leq j < r \}$ with total order
\[
1_0 < \cdots < n_0 < 1_1 < \cdots < n_1 < \cdots < 1_{r-1} < \cdots < n_{r-1}.
\]
The \emph{colored permutation group} $\colpermrn$ is the group of all bijections $\pi: [n]_{(r)} \rightarrow [n]_{(r)}$ such that $\pi(i_j) = k_l$ implies that $\pi(i_{j+a}) = k_{l+a}$ for $1 \leq a < r$ with the sums $j+a$ and $l+a$ evaluated modulo $r$.
Elements of $\colpermrn$ are called \emph{$r$-colored permutations}.

All colored permutations are determined by $\pi(i_0)$ for $1 \leq i \leq n$, so we write colored permutations in one-line notation as $\pi = \pi(1_0) \pi(2_0) \cdots \pi(n_0)$.
In practice, we often simplify this notation and write $\pi(i)$ in place of $\pi(i_0)$.
Every colored permutation is built from an underlying permutation in $\mathfrak{S}_{n}$ obtained from $\pi$ by ignoring the subscripts (colors) of the letters.
We denote this permutation by $|\pi| = |\pi(1)| \cdots |\pi(n)|$ where $|x_k| = x$ for every $x_k \in \rfoldx$.
Define the color map $\epsilon$ by $\epsilon(x_k) = k$ for all $x_k \in \rfoldx$.
For example, the colored permutation $\pi = 2_0 1_3 3_1 5_2 4_2$ is an element of $G_{4,5}$.
Here $|\pi| = 21354$ and $\epsilon(5_2) = 2$.

The group operation in $\colpermrn$ is function composition.
In one-line notation, this means that if $\pi(i) = j_k$ and $\sigma(j) = l_p$, then $(\sigma \pi)(i) = l_{k+p}$ where the sum is evaluated modulo $r$.
As a group, $\colpermrn$ is isomorphic to the wreath product $\Z_r \wr \symn$ where $\Z_r$ denotes the finite cyclic group of order $r$.
Note that colored permutation composition depends on $r$.
For example, if $\pi = 2_0 1_3 3_1 5_2 4_2$ and $\sigma = 3_1 1_1 5_0 2_1 4_3$ are colored permutations in $G_{4,5}$, then $\sigma \pi = 1_1 3_0 5_1 4_1 2_3$.
If instead $\pi, \sigma \in G_{5,5}$, then $\sigma \pi = 1_1 3_4 5_1 4_0 2_3$.

Define the \emph{descent set} $\Des(\pi)$ of a colored permutation $\pi$ written in one-line notation to be the set of all $i \in [n]$ such that $\pi(i) > \pi(i+1)$ with respect to the total order on $[n]_{(r)}$ and with $\pi(n+1) = 0_1$.
Thus $n \in \Des(\pi)$ if and only if $\epsilon(\pi(n)) \neq 0$.
The \emph{descent number} of $\pi$ is $\des(\pi) = |\Des(\pi)|$.
In preparation for later results, we also define the \emph{internal descent set} $\intDes(\pi) = \Des(\pi) \cap [n-1]$ and let $\intdes(\pi) = |\intDes(\pi)|$.
For example, if $\pi = 3_1 1_1 4_0 2_3$, then $\Des(\pi) = \{1,2,4\}$ with $\des(\pi) = 3$.
Also, $\intDes(\pi) = \{1,2\}$ and $\intdes(\pi) = 2$.
Note that this definition of the descent set of a colored permutation is equivalent to that of Steingr{\'{\i}}msson in~\cite{Steingrimsson1994}.

\newpage

\section{Colored Posets}\label{sec: colored posets}

Other authors have altered and extended poset definitions to link linear extensions of posets with specific permutation statistics.
In \cite{ChowThesis2001}, Chow defined signed posets to count signed permutations by the number of descents.
Hsiao and Petersen later defined colored posets in \cite{HsiaoPetersen2010} to obtain results about the Mantaci-Reutenauer algebra.
We give an alternate definition of colored posets that provides information about the descent numbers of colored permutations.

We define an \emph{$r$-colored poset} $P$ to be a partially ordered finite subset of $\{0_1,0_2,\ldots,0_{r-1}\}  \cup  [n]_{(r)}$ such that the elements from $[n]_{(r)}$ all have distinct absolute values. We denote the partial order on $P$ by $\pless$ and require that $r$-colored posets contain the chain $0_1 \pless 0_2 \pless \cdots \pless 0_{r-1}$.
Figure~\ref{fig: 4-colored poset} is the Hasse diagram of the $r$-colored poset $\{0_1, 0_2, 0_3, 1_0, 2_1, 3_1\}$ with partial order $0_1 \pless 0_2 \pless 1_0 \pless 3_1 \pless 0_3$ and $2_1 \pless 1_0$.
\begin{figure}[htbp]
\begin{center}
\[\xymatrix @R=10pt @C=10pt { & & & 0_3 \\ & & 3_1\ar@{-}[ur] & \\  & & 1_0\ar@{-}[u] & \\  & 0_2\ar@{-}[ur] & & 2_1\ar@{-}[ul]  \\  0_1\ar@{-}[ur] & & &  }\] 
\caption{The Hasse diagram of the $4$-colored poset with partial order $0_1 \pless 0_2 \pless 1_0 \pless 3_1 \pless 0_3$ and $2_1 \pless 1_0$. \label{fig: 4-colored poset}}
\end{center}
\end{figure}
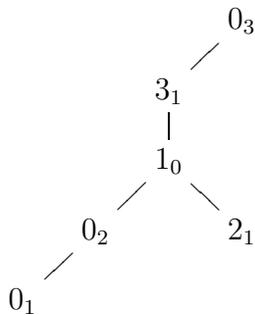

We define colored linear extensions of these colored posets in two stages.
First, given words $w_1$, $w_2$, \ldots, $w_j$ with disjoint letters, let $\shuffle{w_1,w_2,\ldots, w_j}$ denote the set of all shuffles of $w_1$, $w_2$, \ldots, $w_j$.
Thus $\shuffle{w_1,w_2,\ldots, w_j}$ is the set of all words on the union of the letters of $w_1$, $w_2$, \ldots, $w_j$ that contain each of the $w_i$ as subwords when read from left to right.
Note that $\shuffle{w} = \{w\}$.
For every $\pi \in \colpermrn$, let $\shuffle{\pi,0}$ denote the set of shuffles of $\pi$ (written in one-line notation) with the word $0_10_2\cdots 0_{r-1}$.
We define the set of \emph{anchored $r$-colored permutations}, denoted by $\anchoredcolpermrn$, by
\[
\anchoredcolpermrn = \coprod_{\pi \in \colpermrn} \shuffle{\pi,0}.
\]
The set of \emph{linear extensions} of $P$, denoted by $\Lin(P)$, is defined to be the set of all $w \in \anchoredcolpermrn$ such that whenever $x \pless y$, then $x$ is to the left of $y$ in $w$.

We know that each $w \in \anchoredcolpermrn$ is a shuffle of a permutation $\pi \in \colpermrn$ with the word $0_1\cdots 0_{r-1}$.
The $0_i$ letters split the permutation $\pi$ into $r$ (possibly empty) subwords.
Number these subwords $0,1,\ldots,r-1$ from left to right.
For $i=0,\ldots,r-1$, define $\pi_i$ to be the $i$th subword with $i$ subtracted from the color of each letter modulo $r$.
We then define the colored linear extensions of $w \in \anchoredcolpermrn$ by $\collin(w) = \shuffle{\pi_0,\ldots,\pi_{r-1}}.$
Finally, we define the \emph{colored linear extensions} of a colored poset $P$ by
\[
\collin(P) = \coprod_{w \in \Lin(P)} \collin(w) = \coprod_{w \in \Lin(P)} \shuffle{\pi_0,\ldots,\pi_{r-1}}.
\]

\begin{ex}
Let $P$ be the $4$-colored poset in Figure~\ref{fig: 4-colored poset}.
Then
\[
\Lin(P) = \{0_1 0_2 2_1 1_0 3_1 0_3, 0_1 2_1 0_2 1_0 3_1 0_3, 2_1 0_1 0_2 1_0 3_1 0_3\}
\]
and hence
\begin{align*}
\collin(P) &= \shuffle{2_3 1_2 3_3} \cup \shuffle{2_0,1_23_3} \cup \shuffle{2_1, 1_2 3_3} \\
&= \{1_22_03_3, 1_22_13_3, 1_23_32_0, 1_23_32_1, 2_01_23_3, 2_11_23_3, 2_31_23_3\}.
\end{align*}
\end{ex}


\section{Colored $P$-partitions}\label{sec: colored p-partitions}

In \cite{Knuth1970}, Knuth first defined $P$-partitions (see \cite{Gessel2014} for a history of the development of $P$-partitions).
These maps can be used to derive results about the descent set and descent number of permutations (Stanley provides an excellent introduction in Chapter 3 of \cite{Stanley2012}).
Just as with posets, many authors have altered and extended the definition of $P$-partitions to suit their needs.
Stembridge defined enriched $P$-partitions in \cite{Stembridge1997} to study peaks of permutations.
In \cite{ChowThesis2001}, Chow defined signed posets and signed $P$-partitions to count signed permutations by the number of descents.
Later, Petersen in \cite{Petersen2005} altered Chow's definition to study augmented descents of signed permutations.

Let $\pos$ denote the set of positive integers and $\N$ denote the set of nonnegative integers.
The obvious extension of Chow's definition to $\colpermrn$ is to define colored $P$-partitions as maps $f:P \rightarrow \N_{(r)}$ such that if $x \pless y$, then $f(x) \leq f(y)$ in $\N_{(r)}$ with some added restrictions.
We also want to require that if $x \pless y$ and $x>y$ in $[0,n]_{(r)}$, then $f(x) < f(y)$ in $\N_{(r)}$, and if $f(x_j) = k_l$, then $f(x_{j+a}) = k_{l+a}$ for $a=1,\ldots,r-1$.
This approach unfortunately fails because these conditions may be contradictory.
As colors cycle through $\Z_r$, the relationship between $x$ and $y$ can change.
For example, if $\pi = 2_0 1_1$ is a colored permutation in $G_{3,2}$ and we define a poset $P(\pi)$ by $2_0 \pless 1_1$, then we would allow $f(2_0) = f(1_1)$.
This equality would imply that $f(2_2) = f(1_0)$, which would not be allowed since $2_2 > 1_0$ in $\pos_{(r)}$.
To overcome this shifting relationship, we must alter our definition of colored $P$-partitions.

Let $X = \{x_0, x_1, \ldots, x_{\infty}\}$ be a totally ordered countable set with a maximal element and total order $x_0 < x_1 < \cdots < x_{\infty}$.
Let $P$ be an $r$-colored poset.
A \emph{colored $P$-partition} is a function $f: P \rightarrow \rfoldx$ such that:
\begin{enumerate}
\item[(i)] $f(0_k) = (k,x_{0})$ for every $0_k \in P$,
\item[(ii)] $f(a) \leq f(b)$ if $a \pless b$,
\item[(iii)] if $f(a),f(b) \in X_k$, then $f(a) < f(b)$ if $a \pless b$ and $|a|_{\epsilon(a)-k} > |b|_{\epsilon(b)-k}$,
\item[(iv)] if $f(a) = (k,x_{\infty})$, then $\epsilon(a) = k$.
\end{enumerate}
Condition~$\mathrm{(iii)}$ says that if $f(a)$ and $f(b)$ are both mapped into the same color block $X_k$, then the inequality between them is weak or strict as determined by their shifted relationship, not their relationship in the one-line notation of $\pi$.
Condition~$\mathrm{(iv)}$ implies that only an element from $P$ of color $k$ can map to the maximal element in $X_k$.

Let $\mathcal{A}(P)$ denote the set of all colored $P$-partitions.
We can represent every colored permutation $\pi \in \colpermrn$ by the colored poset
\[
\pi(1) \pless \pi(2) \pless \cdots \pless \pi(n) \pless 0_1 \pless \cdots \pless 0_{r-1}.
\]
Note that $\pi$ is the only colored linear extension of this colored poset.
For notational ease, we call these colored $P$-partitions simply \emph{colored $\pi$-partitions} and denote the set of all such colored $\pi$-partitions by $\A(\pi)$.
Similarly, if $w\in \anchoredcolpermrn$, then we denote the set of all \emph{colored $w$-partitions} by $\A(w)$.
Note that $w$ can be represented by the following poset, denoted by $P(w)$,
\[
w(1) \pless w(2) \pless \cdots \pless w(n+r-1).
\]
Here $w$ is the only linear extension of this colored poset.
We have the following theorem which resembles the fundamental theorems from other authors (see \cite{ChowThesis2001, HsiaoPetersen2010, Petersen2005, Stanley1972, Stembridge1997}), so we call it the \emph{Fundamental Theorem of Colored $P$-partitions}.

\begin{thm}[FTCPP]
There is a bijection between the set of all colored $P$-partitions $\A(P)$ and the disjoint union of the sets of colored $\pi$-partitions over all colored linear extensions $\pi$ of $P$.
That is,
\[
\A(P) \leftrightarrow \coprod_{\pi \in \collin(P)} \A(\pi).
\]
\end{thm}

\begin{proof}
We first prove that
\[
\A(P) = \coprod_{w \in \Lin(P)} \A(w).
\]
The proof is by induction on the number of incomparable pairs $(a, b)$.
If $a$ and $b$ are incomparable in $P$, then let $P_{ab}$ denote the poset obtained from $P$ by adding the relation $a \pless b$.
Similarly, $b \pless a$ in $P_{ba}$.
While it is clear that $\Lin(P) = \Lin(P_{ab}) \sqcup \Lin(P_{ba})$, we must show that $\mathcal{A}(P) = \mathcal{A}(P_{ab}) \sqcup \mathcal{A}(P_{ba})$.

If we assume that $|a| \neq 0$ and $b = 0_k$ for some $k\in [r-1]$, then $\A(P_{ab})$ contains all colored $P$-partitions $f$ such that $f(a) \in X_l$ for $l < k$.
Also, $\A(P_{ba})$ contains all colored $P$-partitions $f$ such that $f(a) \in X_l$ for $l \geq k$.
Thus $\mathcal{A}(P) = \mathcal{A}(P_{ab}) \sqcup \mathcal{A}(P_{ba})$ and, by induction, we can assume that the set $\{a, 0_1, \ldots, 0_{r-1}\}$ is totally ordered for every $a \in P$ with $|a| \neq 0$.

We next consider incomparable pairs $(a, b)$ with $|a|\neq 0 $ and $|b|\neq 0$.
Let $f \in \A(P)$. 
By the previous argument, we can assume that $f(a), f(b) \in X_k$ for some $k\in [0,r-1]$.
If $|a|_{\epsilon(a)-k} < |b|_{\epsilon(b)-k}$, then equations \eqref{eq: FTCPP less than case 1}, \eqref{eq: FTCPP less than case 2}, and \eqref{eq: FTCPP less than case 3} give the decomposition of $\A(P)$ into $\A(P_{ab})$ and $\A(P_{ba})$.
Note that if $\epsilon(a) \neq k$ and $\epsilon(b) = k$, then $|a|_{\epsilon(a)-k} \nless |b|_{\epsilon(b)-k}$.
\begin{align}
 & & & \underline{\A(P_{ab})} & & \underline{\A(P_{ba})} \notag \\
\epsilon(a) = k, \; & \epsilon(b) = k & f(a) \leq & f(b) \leq (k,x_{\infty}) & f(b) < & f(a) \leq (k,x_{\infty}) \label{eq: FTCPP less than case 1} \\
\epsilon(a) = k, \; & \epsilon(b) \neq k & f(a) \leq & f(b) < (k,x_{\infty}) & f(b) < & f(a) \leq (k,x_{\infty})\label{eq: FTCPP less than case 2} \\
\epsilon(a) \neq k, \; & \epsilon(b) \neq k & f(a) \leq & f(b) < (k,x_{\infty}) & f(b) < & f(a) < (k,x_{\infty}) \label{eq: FTCPP less than case 3}
\end{align}

Similarly, if $|a|_{\epsilon(a)-k} > |b|_{\epsilon(b)-k}$, then equations \eqref{eq: FTCPP greater than case 1}, \eqref{eq: FTCPP greater than case 2}, and \eqref{eq: FTCPP greater than case 3} give the decomposition of $\A(P)$ into $\A(P_{ab})$ and $\A(P_{ba})$.
Note that if $\epsilon(a) = k$ and $\epsilon(b) \neq k$, then $|a|_{\epsilon(a)-k} \ngtr |b|_{\epsilon(b)-k}$.
\begin{align}
 & & & \underline{\A(P_{ab})} & & \underline{\A(P_{ba})} \notag\\
\epsilon(a) = k, \; & \epsilon(b) = k & f(a) < & f(b) \leq (k,x_{\infty}) & f(b) \leq & f(a) \leq (k,x_{\infty}) \label{eq: FTCPP greater than case 1} \\
\epsilon(a) \neq k, \; & \epsilon(b) = k & f(a) < & f(b) \leq (k,x_{\infty}) & f(b) \leq & f(a) < (k,x_{\infty}) \label{eq: FTCPP greater than case 2} \\
\epsilon(a) \neq k, \; & \epsilon(b) \neq k & f(a) < & f(b) < (k,x_{\infty}) & f(b) \leq & f(a) < (k,x_{\infty}) \label{eq: FTCPP greater than case 3}
\end{align}
Thus $\mathcal{A}(P) = \mathcal{A}(P_{ab}) \sqcup \mathcal{A}(P_{ba})$ and we have $\A(P) = \coprod_{w \in \Lin(P)} \A(w)$.

The final step in the proof of the FTCPP is to describe a bijection between $\A(w)$ and the disjoint union of the sets of colored $\pi$-partitions over all colored linear extensions $\pi$ of $P(w)$.
That is,
\[
\A(w) \leftrightarrow \coprod_{\pi \in \collin(P(w))} \A(\pi).
\]
Suppose $w$ has underlying colored permutation $\sigma$.
First, given a colored $\pi$-partition $g$ for some $\pi \in \collin(P(w))$, we define a colored $w$-partition $f$.
Since $\pi \in \collin(P(w))$, we know that $\pi(k)$ is a letter in $\sigma_i$ for some $i\in [0,r-1]$.
Hence there exists an $l\in [n]$ such that $\pi(k) = |\sigma(l)|_{\epsilon(\sigma(l))-i}$.
We then define $f(\sigma(l)) = |g(\pi(k))|_{i}$.
Repeating this for every $k \in [n]$ defines the desired colored $w$-partition $f$.
Next, given a colored $w$-partition $f$, we define a colored $\pi$-partition $g$ for some $\pi \in \collin(P(w))$.
Since $\pi(k) = |\sigma(l)|_{\epsilon(\sigma(l))-i}$, we set $g(\pi(k)) = |f(\sigma(l))|_0$.
Repeating this for every $k\in [n]$ defines a colored $\pi$-partition $g$ and the composition of these two bijections in either order gives the identity.
\end{proof}

The following example illustrates both stages of the bijection found in the proof of the FTCPP.

\begin{ex}
Let $P$ be the 4-colored poset in Figure~\ref{fig: 4-colored poset} and let $f \in \mathcal{A}(P)$.
We first see that $f(1_0), f(3_1) \in X_2$.
Next we have $1_{0-2} = 1_2 < 3_3 = 3_{1-2}$, so $f(1_0) \leq f(3_1)$.
Then, since $\epsilon(3_1) \neq 2$, we have $f(3_1) < (2,x_{\infty})$.
Taken together, we have $(2,x_0) \leq f(1_0) \leq f(3_1) < (2,x_{\infty})$.
By examining the image of $2_1$, $\A(P)$ can be decomposed into the three sets.
\begin{enumerate}
\item If $f(2_1) \in X_2$, then we have the set $\A(0_1 0_2 2_1 1_0 3_1 0_3)$ which is equal to
\[
\{f \, : \, (2,x_{0}) \leq f(2_1) < f(1_0) \leq f(3_1) < (2,x_{\infty}) \}.
\]
Here $f(2_1) < f(1_0)$ because $2_1 \pless 1_0$ and $2_{1-2} = 2_3 > 1_2 = 1_{0-2}$ in $\pos_{(r)}$.

\item If $f(2_1) \in X_1$, then we have the set $\A(0_1 2_1 0_2 1_0 3_1 0_3)$ which is equal to
\[
\{f \, : \, (1,x_0) \leq f(2_1) \leq (1,x_{\infty}) \text{ and } (2,x_{0}) \leq f(1_0) \leq f(3_1) < (2,x_{\infty})\}.
\]
Here $f(2_1) \leq (1,x_{\infty})$ because $\epsilon(2_1) = 1$.

\item If $f(2_1) \in X_0$, then we have the set $\A(2_1 0_1 0_2 1_0 3_1 0_3)$ which is equal to
\[
\{f \, : \,  (0,x_0) \leq f(2_1) < (0,x_{\infty}) \text{ and } (2,x_{0}) \leq f(1_0) \leq f(3_1) < (2,x_{\infty})\}.
\]
Here $f(2_1) < (0,x_{\infty})$ because $\epsilon(2_1) \neq 0$.
\end{enumerate}
Thus we have shown that
\[
\A(P) = \A(0_1 0_2 2_1 1_0 3_1 0_3) \sqcup \A(0_1 2_1 0_2 1_0 3_1 0_3) \sqcup \A(2_1 0_1 0_2 1_0 3_1 0_3) = \coprod_{w \in \Lin(P)} \A(w).
\]
Next we examine the sets $\A(w)$ for $w\in \Lin(P)$ and illustrate the bijection between $\A(w)$ and $\coprod_{\pi \in \collin(P(w))} \A(\pi)$.
It is easy to see the bijection between the two sets
\begin{align*}
\A(0_1 0_2 2_1 1_0 3_1 0_3) &=  \{f \, : \, (2,x_{0}) \leq f(2_1) < f(1_0) \leq f(3_1) < (2,x_{\infty}) \}
\intertext{and}
\A(2_3 1_2 3_3) &=  \{f \, : \, (0,x_{0}) \leq f(2_3) < f(1_2) \leq f(3_3) < (0,x_{\infty}) \}.
\end{align*}
If we define $\coprod_{\pi \in \shuffle{2_0,1_23_3}} \A(\pi)$ to be
\[
\{f \, : \, (0,x_0) \leq f(2_0) \leq (0,x_{\infty}) \text{ and } (0,x_{0}) \leq f(1_2) \leq f(3_3) < (0,x_{\infty}) \},
\]
we see the bijection
\[
\A(0_1 2_1 0_2 1_0 3_1 0_3) \leftrightarrow \A(1_22_03_3) \sqcup \A(1_23_32_0) \sqcup \A(2_01_23_3).
\]
Note that every $f\in \A(0_1 2_1 0_2 1_0 3_1 0_3)$ with $f(2_1) = (1,x_{\infty})$ is mapped to an element of $\A(1_23_32_0)$.
Finally, the bijection
\[
\A(2_1 0_1 0_2 1_0 3_1 0_3) \leftrightarrow \A(1_22_13_3) \sqcup \A(1_23_32_1) \sqcup \A(2_11_23_3)
\]
is analogous.
\end{ex}

The previous notation can be cumbersome and we will not need such generality to prove the existence of the colored Eulerian descent algebra.
Thus, for the rest of the paper, we set $X = [0,j]$ and define the order polynomial $\Omega_{P} (j)$ to be the number of colored $P$-partitions with parts in $\rversion{[0,j]}$.
Hence $\Omega_{P} (j) = | \A(P) |$.
Also define $\Omega_{\pi}(j)$ to be the number of colored $\pi$-partitions with parts in $\rversion{[0,j]}$.
We then have the following corollary of the Fundamental Theorem of Colored $P$-partitions.

\begin{cor}
\label{cor: order poly summation}
\[
\Omega_{P} (j) = \sum_{\pi \in \collin(P)} \Omega_{\pi} (j).
\]
\end{cor}

Corollary~\ref{cor: order poly summation} allows us to count colored $P$-partitions by instead counting the much simpler colored $\pi$-partitions.
We know that $\Omega_\pi (j)$ counts colored $\pi$-partitions $f$ with
\[
0_0 \leq f(\pi(1)) \leq f(\pi(2)) \leq \cdots \leq f(\pi(n)) \leq j_0,
\]
where $f(\pi(i)) < f(\pi(i+1))$ if $\pi(i) > \pi(i+1)$.
Note that $f(\pi(n)) < j_0$ if $n \in \Des(\pi)$.
Using a well-known technique, there exists a bijection between these sequences and sequences for which every inequality is weak.
We change strict inequalities to weak inequalities by subtracting $1$ from all terms to the right of the strict inequality.
Doing this for every strict inequality, we can compute $\Omega_{\pi}(j)$ by instead counting the number of sequences $(s_1, s_2, \ldots, s_n)$ where
\[
0_0 \leq s_1 \leq s_2 \leq \cdots \leq s_n \leq j_0 - \des(\pi).
\]
Hence
\begin{equation}\label{eq: colored order poly}
\Omega_{\pi}(j) = \mchoose{j-\des(\pi)+1}{n} = \binom{j+n-\des(\pi)}{n},
\end{equation}
where $\mchoose{a}{b}$ denotes the ``multichoose'' function, i.e., the number of ways to choose $b$ elements from a set of size $a$ allowing repetition.
Hence we have
\[
\sum_{j \geq 0} \Omega_{\pi}(j) t^j = \sum_{j \geq 0} \binom{j+n-\des(\pi)}{n} t^j = t^{\des(\pi)} \sum_{j \geq 0} \binom{j+n}{n} t^j = \frac{t^{\des(\pi)}}{(1-t)^{n+1}}.
\]
Combining this result with Corollary~\ref{cor: order poly summation} tells us that
\[
\sum_{j \geq 0} \Omega_{P}(j)  t^j = \frac{\dsum_{\pi \in \collin(P)} t^{\des(\pi)}}{(1-t)^{n+1}}.
\]

Continuing towards the goal of being able to compute $\Omega_{P}(j)$ for colored posets of various forms, let $P_1 \sqcup P_2$ denote the disjoint union of the two colored posets $P_1$ and $P_2$.
If $\{|a| \, : \, a \in P_1, |a| \neq 0 \}$ and $\{|b| \, : \, b \in P_2, |b| \neq 0 \}$ are disjoint, then the map which sends a colored $(P_1 \sqcup P_2)$-partition $f$ to the ordered pair $(g,h)$ where $g = f|_{P_1}$ and $h = f|_{P_2}$ gives us the following theorem.

\begin{thm}
There is a bijection between the set $\mathcal{A}(P_1 \sqcup P_2)$ and the cartesian product $\mathcal{A}(P_1) \times \mathcal{A}(P_2)$.
That is, $\mathcal{A}(P_1 \sqcup P_2) \leftrightarrow \mathcal{A}(P_1) \times \mathcal{A}(P_2)$.
\end{thm}

\begin{cor}
\label{cor: order poly product}
\[
\Omega_{P_1 \sqcup P_2} (j) = \Omega_{P_1}(j) \cdot \Omega_{P_2}(j).
\]
\end{cor}

Certain colored posets yield nice results about the colored permutation groups $\colpermrn$.
If $P$ is the disjoint union of the singleton element $a_0$ for $a \in \pos$ and the chain $0_1 \pless 0_2 \pless \cdots \pless 0_{r-1}$, then $\collin(P) = \{a_k \, : \, k \in [0,r-1]\}$.
Also, $\Omega_P(j) = j+1 + (r-1)j  = rj+1$ since $\Omega_P(j)$ counts maps from $a_0$ to the set
\[
0_0 < \cdots < j_0  < 0_1 < \cdots < (j-1)_1 < \cdots < 0_{r-1} < \cdots < (j-1)_{r-1}.
\]
Note that $\collin(P)$ and $\Omega_P(j)$ are invariant with respect to the choice of color for $a\in \pos$.

If $P$ is the disjoint union of the antichain on $1_0, 2_0, \ldots, n_0$ together with the chain $0_1 \pless 0_2 \pless \cdots \pless 0_{r-1}$, then $\collin(P) = \colpermrn$.
Furthermore, by Corollary~\ref{cor: order poly product}, we know that $\Omega_P(j) = (rj+1)^n$.
This gives an alternate proof of the following theorem of Steingr{\'{\i}}msson \cite[Theorem 17]{Steingrimsson1994}.

\begin{thm}[Steingr{\'{\i}}msson]\label{Col Eul Poly}
\[
\sum_{j \geq 0} (rj+1)^n t^j = \frac{\dsum_{\pi \in \colpermrn} t^{\des(\pi)}}{(1-t)^{n+1}}.
\]
\end{thm}

To prove the existence of the colored Eulerian descent algebra, we will need to better understand the following colored poset.
Let $P(\pi)$ denote the disjoint union of the chains
\[
0_1 \pless 0_2 \pless \cdots \pless 0_{r-1} \quad \text{and} \quad \pi(1) \pless \pi(2) \pless \cdots \pless \pi(n).
\]
At first glance it seems difficult to compute $\Omega_{P(\pi)}(j)$.
However, if $\pi$ is monochromatic, i.e., if $\epsilon(\pi(1)) = \cdots = \epsilon(\pi(n))$, then $\Omega_{P(\pi)}(j)$ is easily computable and we have the following lemma.

\begin{lem}\label{lem: mono order poly}
If $\pi \in \colpermrn$ is monochromatic, then
\[
\Omega_{P(\pi)}(j) = \binom{rj+n-\intdes(\pi)}{n}.
\]
\end{lem}

\begin{proof}
Suppose $\pi$ is monochromatic, and let $f\in \mathcal{A}(P(\pi))$.
The relationship between $f(\pi(i))$ and $f(\pi(i+1))$ in $X_k$ is the same for all $k=0,\ldots,r-1$ because $\epsilon(\pi(i))-k = \epsilon(\pi(i+1))-k$ for all values of $k$.
For notational ease, let $\epsilon(\pi) = \epsilon(\pi(1))$.
We also have $f(\pi(i)) \neq j_k$ for $k \neq \epsilon(\pi)$.
Thus $\Omega_{P(\pi)}(j)$ counts maps, $f$, from $P(\pi)$ to the totally ordered set
\[
0_0 < \cdots < (j-1)_0  < \cdots < 0_{\epsilon(\pi)} < \cdots < j_{\epsilon(\pi)} < \cdots < 0_{r-1} < \cdots < (j-1)_{r-1}
\]
such that $f(\pi(i)) \leq f(\pi(i+1))$ and $f(\pi(i)) < f(\pi(i+1))$ if $\pi(i) > \pi(i+1)$.
This image set has size $j+1+ (r-1)j = rj+1$, and a descent at position $n$ does not affect $\Omega_{P(\pi)} (j)$.
Hence
\[
\Omega_{P(\pi)}(j) = \mchoose{rj+1-\intdes(\pi)}{n} = \binom{rj+n-\intdes(\pi)}{n}.
\]
\end{proof}

As it turns out, the previous lemma holds for all $\pi \in \colpermrn$, but to prove this we need an intermediate lemma.
First, we introduce some new notation.
If $\pi$ is not monochromatic, then let $i$ be the largest value such that $\epsilon(\pi(1)) = \cdots = \epsilon(\pi(i))$.
Thus $i$ is the length of the monochromatic run beginning with $\pi(1)$.
Let $b=\epsilon(\pi(i+1))$ and define a new permutation $\rho(\pi)$ by setting
\[
\rho(\pi)(s) = \left\{\begin{array}{cl}|\pi(s)|_b, & \text{if } 1\leq s \leq i \\ \pi(s), & \text{if } i < s \leq n. \end{array}\right.
\]
The map $\rho$ simply shifts the color of each of the letters in the initial monochromatic run to $b$ and increases the length of the run.
Note that doing so may or may not change the number of descents.

Next we define a useful family of bijections.
For $a=1,2,\ldots,r-1$, we define $\phi_a: [0,j]_{(r)} \rightarrow [0,j]_{(r)}$ by
\[
\phi_a(x_c) = \left\{\begin{array}{rl} x_c, & \text{if } x_c < j_{a-1} \\ 0_{a}, & \text{if } x_c = j_{a-1}  \\ (x+1)_{a}, &  \text{if } 0_a \leq x_c < j_a \\ j_{a-1}, & \text{if } x_c = j_{a} \\ x_c, & \text{if } x_c > j_{a} . \end{array}\right.
\]
A general image of one of these bijections is given in Figure~\ref{fig: MonoRunShiftForward}.
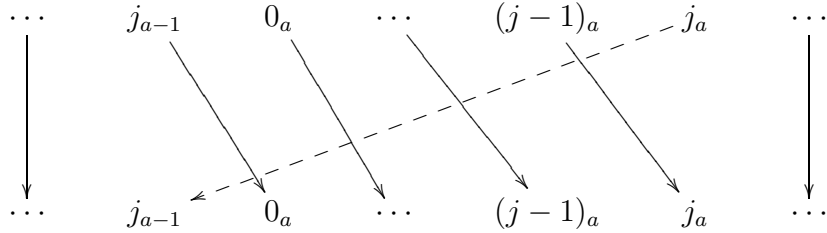
\begin{figure}[htbp]
\begin{center}
\[\xymatrix{ \cdots \ar@{->}[dd] & j_{a-1} \ar@{->}[ddr] & 0_{a} \ar@{->}[ddr] & \cdots \ar@{->}[ddr]  & (j-1)_{a} \ar@{->}[ddr]  & j_a \ar@{-->}[ddllll]  & \cdots \ar@{->}[dd] \\ &&&&&& \\ \cdots & j_{a-1} & 0_{a} & \cdots & (j-1)_{a} & j_{a} & \cdots }\] 
\caption{The bijection $\phi_{a}: [0,j]_{(r)} \rightarrow [0,j]_{(r)}$. \label{fig: MonoRunShiftForward}}
\end{center}
\end{figure}
 
\begin{lem}\label{lem: MonoRunShift}
If $\des(\pi) = \des(\rho(\pi))$, then $\Omega_{P(\pi)}(j) = \Omega_{P(\rho(\pi))}(j)$.
\end{lem}

\begin{proof}
This proof is notationally unpleasant but is really just a careful description of a bijection between the sets $\A(P(\pi))$ and $\A(P(\rho(\pi)))$.
This bijection is obtained through the $\phi_a$ maps defined earlier.
Suppose that the length of the initial monochromatic run is $i$ and let $a = \epsilon(\pi(i))$ and $b = \epsilon(\pi(i+1))$.

Case 1:
Assume $a < b$.
Define a new colored permutation $\sigma$ by setting $\sigma(s) = |\pi(s)|_{a+1}$ for $1 \leq s \leq i$ and $\sigma(s) = \pi(s)$ for $i < s \leq n$.
If $\des(\pi) = \des(\sigma)$, then we define a bijection between colored $P(\pi)$-partitions and colored $P(\sigma)$-partitions as follows.
For every colored $P(\pi)$-partition $f$, define a colored $P(\sigma)$-partition $g$ by setting $g(\sigma(s)) = \phi_{a+1}(f(\pi(s)))$ for $1 \leq s \leq i$.
If $i < s \leq n$, then define $g(\sigma(s)) = f(\pi(s))$.
This bijection shows that $\Omega_{P(\pi)}(j) = \Omega_{P(\sigma)}(j)$.
This process can be repeated until $\epsilon(\pi(i))+1 = \epsilon(\pi(i+1))$.
Then, as long as doing so does not change the number of descents, it can be repeated once more so that $\sigma = \rho(\pi)$ and thus $\Omega_{P(\pi)}(j) = \Omega_{P(\rho(\pi))}(j)$.

To see why $\phi_{a+1}$ produces the desired bijection, let $f$ be a colored $P(\pi)$-partition and let $g$ be a colored $P(\sigma)$-partition.
First note that for $s \neq i$, if $\pi(s)$ and $\pi(s+1)$ are mapped to $X_k$, then the relation $f(\pi(s)) \sim f(\pi(s+1))$ is the same as the relation $g(\sigma(s)) \sim g(\sigma(s+1))$ when they are both mapped to $X_k$.
That is, they are either both weak or both strict inequalities.
This is because if $s < i$, then $\epsilon(\pi(s)) = \epsilon(\pi(s+1))$ and $\epsilon(\sigma(s)) = \epsilon(\sigma(s+1))$.
If $s > i$, then $\epsilon(\pi(s)) = \epsilon(\sigma(s))$.
Also, if $s \leq i$, then $f(\pi(s)) \neq j_k$ unless $k=a$.
Similarly, $g(\sigma(s)) \neq j_k$ unless $k = a+1$.
If $s> i$, then (since $\epsilon(\pi(s)) = \epsilon(\sigma(s))$) both $f(\pi(s)) \neq j_k$ and $g(\sigma(s)) \neq j_k$ unless $\epsilon(\pi(s)) = k$.

Next consider $\pi(i)$ and $\pi(i+1)$.
We claim that if  $\pi(i)$ and $\pi(i+1)$ are mapped to $X_k$ for $k\neq a+1$, then the relation $f(\pi(i)) \sim f(\pi(i+1))$ is the same as the relation $g(\sigma(i)) \sim g(\sigma(i+1))$ when they are both mapped to $X_k$.
However, if they are both mapped to $X_{a+1}$, then $f(\pi(i)) < f(\pi(i+1))$ and $g(\sigma(i)) \leq g(\sigma(i+1))$.
Lastly, note that $f(\pi(i)) \neq j_k$ unless $k= a$, and $g(\sigma(i)) \neq j_k$ unless $k = a+1$.

Case 2:
Assume $a > b$.
Define a new permutation $\sigma$ by setting $\sigma(s) = |\pi(s)|_{a-1}$ for $1 \leq s \leq i$ and $\sigma(s) = \pi(s)$ for $i < s \leq n$.
If $\des(\pi) = \des(\sigma)$, then we define a bijection between colored $P(\pi)$-partitions and colored $P(\sigma)$-partitions as follows.
For every colored $P(\pi)$-partition $f$, define a colored $P(\sigma)$-partition $g$ by setting $g(\sigma(s)) = \phi_{a}^{-1} (f(\pi(s)))$ for $1 \leq s \leq i$.
If $i < s \leq n$, then define $g(\sigma(s)) = f(\pi(s))$.
This bijection shows that $\Omega_{P(\pi)}(j) = \Omega_{P(\sigma)}(j)$.
This process can be repeated until $\epsilon(\pi(i))-1 = \epsilon(\pi(i+1))$.
Then, as long as doing so does not change the number of descents, it can be repeated once more so that $\sigma = \rho(\pi)$.
Thus $\Omega_{P(\pi)}(j) = \Omega_{P(\rho(\pi))}(j)$.
The details for Case 2 are similar to those of Case 1 and are omitted.
\end{proof}

\begin{ex}
Let $X=\{0,1,2\}$ and $r=3$.
Figure~\ref{fig: MonoRunShiftExample} shows how the bijection from the previous lemma maps colored $P(\pi)$-partitions with $\pi = 2_0 1_0 3_2$ to colored $P(\rho(\pi))$-partitions with $\rho(\pi) = 2_2 1_2 3_2$ via the composition $\phi_2 \phi_1$.

\begin{figure}[htbp]
\begin{center}
\[\xymatrix{ \ar@/_1pc/[dd]_{\phi_1}   & 0_0 \ar@{->}[dd] & 1_0 \ar@{->}[dd] & 2_0 \ar@{->}[ddr] & 0_1 \ar@{->}[ddr] & 1_1 \ar@{->}[ddr] & \times \ar@{-->}[ddlll] & 0_2 \ar@{->}[dd] & 1_2 \ar@{->}[dd] & \times \ar@{-->}[dd] \\
 & & & & & & & & & \\
 \ar@/_1pc/[dd]_{\phi_2} & 0_0 \ar@{->}[dd] & 1_0 \ar@{->}[dd] & \times \ar@{-->}[dd] & 0_1 \ar@{->}[dd] & 1_1 \ar@{->}[dd] & 2_1 \ar@{->}[ddr] & 0_2 \ar@{->}[ddr] & 1_2 \ar@{->}[ddr] & \times \ar@{-->}[ddlll]  \\
 & & & & & & & & & \\
 & 0_0  & 1_0 & \times & 0_1 & 1_1 & \times & 0_2 & 1_2 & 2_2 }\] 
\caption{The bijection $\phi_2 \phi_1: [0,2]_{(3)} \rightarrow [0,2]_{(3)}$. \label{fig: MonoRunShiftExample}}
\end{center}
\end{figure}
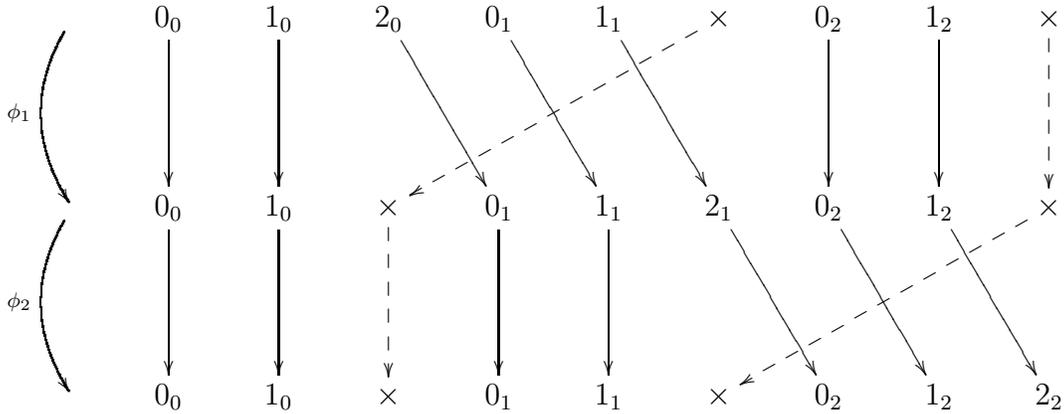

\end{ex}

We are finally ready to compute $\Omega_{P(\pi)}(j)$ for all $\pi \in \colpermrn$.

\begin{thm}\label{thm: P(pi) order poly}
If $\pi \in \colpermrn$, then
\[
\Omega_{P(\pi)}(j) = \binom{rj+n-\intdes(\pi)}{n}.
\]
\end{thm}

\begin{proof}
We first replace $\pi$ with an alternate permutation, $\sigma$, such that $|\sigma|$ has the same descent set as the internal descent set of $\pi$.
The purpose of the standardization is to find a permutation $\sigma$ such that in applying Lemma \ref{lem: MonoRunShift} we can always shift the color of the initial monochromatic run to match the color of the next term without creating a new descent.
We define $\sigma$ using the following inductive algorithm.
If $\pi(i)$ is the smallest letter in $\pi$ (still using the lexicographic order), then define $\sigma(i) = 1_{\epsilon(\pi(i))}$.
Similarly, if $\pi(j)$ is the second smallest letter in $\pi$, then set $\sigma(j) = 2_{\epsilon(\pi(j))}$.
Continue this process until all $n$ letters of $\sigma$ have been defined.

This produces a new permutation $\sigma$ with the same descent set as $\pi$.
To see why, note that $\epsilon(\pi(i)) = \epsilon(\sigma(i))$ for $i = 1,\ldots,n$.
Thus all descents between colors are preserved.
Also, if $\pi$ has a descent at position $i$ with $\epsilon(\pi(i)) = \epsilon(\pi(i+1))$, then $\pi(i) > \pi(i+1)$.
Hence $|\sigma(i)| > |\sigma(i+1)|$, and this descent is also preserved.
Next, note that $\Omega_{P(\pi)}(j) = \Omega_{P(\sigma)}(j)$ because both descent set and the individual colors of the letters are preserved under the standardization.
Repeatedly applying Lemma \ref{lem: MonoRunShift} allows us to shift from $\pi$ to $\tau$ defined by $\tau(i) = |\pi(i)|_{\epsilon(\pi(n))}$ for $i = 1,\ldots, n$.
Since $\tau$ is monochromatic, $\Omega_{P(\tau)}(j)$ is given by Lemma \ref{lem: mono order poly}.
Lastly, we know that $\intdes(\tau) = \intdes(\pi)$ and that $\Omega_{P(\pi)}(j) =  \Omega_{P(\tau)}(j)$.
Hence
\[
\Omega_{P(\pi)}(j) = \binom{rj+n-\intdes(\pi)}{n}.
\]
\end{proof}

We now have a sufficient understanding of colored $P$-partitions and are ready to prove the existence of the colored Eulerian descent algebra.
To do so, we borrow a technique of adding bars to permutations from Gessel in \cite{GesselThesis} and adapt it to our colored posets.


\section{The colored Eulerian descent algebra}\label{sec: colored descent algebra}

In this section, we will prove that the Mantaci-Reutenauer algebra contains a subalgebra induced by descent number with basis elements $C_i$ defined by
\[
C_i = \sum_{\substack{\pi \in \colpermrn \\ \des(\pi) = i}} \pi.
\]
To prove that such an algebra exists, we compute $C_\pi(s,t) = \sum_{\sigma \tau = \pi} s^{\des(\sigma)}t^{\des(\tau)}$ and show that $C_\pi(s,t)$ is a function of $\des(\pi)$.
This gives the desired result because the coefficient of $s^jt^k$ in $C_{\pi}(s,t)$ is the same as the coefficient of $\pi$ in $C_j \cdot C_k$.
Both count the number of pairs $(\sigma, \tau)$ such that $\sigma \tau = \pi$, $\des(\sigma) = j$, and $\des(\tau) = k$.
If $C_{\pi_1}(s,t) = C_{\pi_2}(s,t)$ whenever $\des(\pi_1) = \des(\pi_2)$, then the coefficients of $\pi_1$ and $\pi_2$ will always be equal in the product $C_j \cdot C_k$.
Hence this product can always be written as a linear combination of the $C_i$ terms.
Our main result is Theorem~\ref{Colored Descent Alg} which shows that the multivariate polynomials $C_{\pi}(s,t)$ are functions of $\des(\pi)$.

\begin{thm}\label{Colored Descent Alg}
For every $\pi \in \colpermrn$,
\begin{equation}\label{eq: col des algebra}
\sum_{j,k \geq 0} \binom{ rjk+j+k+n-\des(\pi)}{n} s^j t^k = \frac{\dsum_{\sigma\tau = \pi} s^{\des(\sigma)} t^{\des(\tau)}}{(1-s)^{n+1}(1-t)^{n+1}}.
\end{equation}
\end{thm}

Before we can prove Theorem~\ref{Colored Descent Alg}, we need to examine two special posets.
For every $I \subseteq [n]$ and $\pi \in \colpermrn$, define the \emph{colored zig-zag poset} $\zigipi$ by setting $\pi(i) \pless \pi(i+1)$ if $i \notin I$ and $\pi(i) \pgreat \pi(i+1)$ if $i \in I$ with $\pi(n+1) = 0_1$.
We still require that $0_1 \pless 0_2 \pless \cdots \pless 0_{r-1}$ but will not mention this explicitly except where necessary.

\begin{lem}\label{Col Zig-Zag Extensions}
Let $\sigma, \pi \in \colpermrn$ and $I \subseteq [n]$.
Then $\sigma \in \collin(\zigipi)$ if and only if $\Des(\sigma^{-1}\pi) = I$.
\end{lem}

\begin{proof}
Let $\sigma \in \colpermrn$.
Suppose $\pi(i) = A_a$ and $\pi(i+1) = B_b$.
If $i \notin I$, then we have $A_a \pless B_b$.
Thus there exist $c,d,C,D$ such that $\sigma^{-1}(A_{a-c}) = C_0$ and $ \sigma^{-1}(B_{b-d}) = D_0$ with $c\leq d$.
Also, if $c = d$, then $C < D$.
Hence  $\sigma^{-1}\pi(i) = C_c < D_d = \sigma^{-1}\pi(i+1)$.
Similarly, if $i \in I$, then $\sigma^{-1}\pi(i)  > \sigma^{-1}\pi(i+1)$.
This also shows that $\sigma^{-1}\pi(n) = C_c$, where $c=0$ if and only if $n \notin I$.
Hence $n \in \Des(\sigma^{-1}\pi)$ if and only if $n \in I$.
\end{proof}

\begin{ex}
If $r=3$, $n=3$, $\pi = 2_1 1_2 3_2$, and $I = \{1\}$, then $\zigipi$ is the poset $2_1 \pgreat 1_2 \pless 3_2 \pless 0_1$.
We have
\begin{align*}
\collin(\zigipi) &= \shuffle{1_2 2_1 3_2} \cup \shuffle{1_2 3_2 2_1} \cup \shuffle{1_23_2,2_0} \cup \shuffle{1_23_2,2_2} \\
&= \{1_2 2_0 3_2, 1_2 2_1 3_2, 1_2 2_2 3_2, 1_2 3_2 2_0, 1_2 3_2 2_1, 1_2 3_2 2_2, 2_0 1_2 3_2, 2_2 1_2 3_2\}
\end{align*}
 and
 \[
 \coprod_{\sigma \in \collin(\zigipi)} \{\sigma^{-1}\pi  \} = \{2_0 1_0 3_0, 3_0 1_0 2_0, 1_1 2_0 3_0, 2_1 1_0 3_0, 3_1 1_0 2_0, 1_2 2_0 3_0, 2_2 1_0 3_0, 3_2 1_0 2_0\}.
 \]
Here $\Des(\sigma^{-1}\pi) = \{1\}$ for every $\sigma \in \collin(\zigipi)$.
\end{ex}

Next we consider a second family of posets.
For every $I \subseteq [n]$ and $\pi \in \colpermrn$, define the \emph{colored chain poset} $\chainipi$ by setting $\pi(i) \pless \pi(i+1)$ if $i \notin I$ with $\pi(n+1) = 0_1$.
The following lemma about linear extensions of colored chain posets is similar to that for colored zig-zag posets and follows from the same proof.

\begin{lem}\label{Col Chain Extensions}
Let $\sigma, \pi \in \colpermrn$ and $I \subseteq [n]$.
Then $\sigma \in \collin(\chainipi)$ if and only if $\Des(\sigma^{-1}\pi) \subseteq I$.
\end{lem}

\begin{ex}
If $r=3$, $n=3$, $\pi = 2_1 1_2 3_2$, and $I = \{1\}$, then $\chainipi$ is the poset $ 1_2 \pless 3_2 \pless 0_1$ with no relation between $2_1$ and any other elements.
Then $\collin(\chainipi) = \collin(\zigipi) \cup \{2_1 1_2 3_2\}$.
Hence
\[
\{\sigma^{-1}\pi \, | \, \sigma \in \collin(\chainipi)\} = \{\sigma^{-1}\pi \, | \, \sigma \in \collin(\zigipi)\} \cup \{1_0 2_0 3_0\}
\]
and $\Des(1_0 2_0 3_0) = \varnothing$.
\end{ex}

Next we define barred versions of both colored zig-zag posets and colored chain posets.
The barring technique is both powerful and flexible and is borrowed from Gessel \cite{GesselThesis}.
Adding bars to permutations produces very elegant results which count permutations (or multiset permutations) by descent number and major index.
For examples of this counting technique (and extensions to signed and colored permutations) see \cite{GesselThesis, MoynihanThesis}.

First, a \emph{barred colored zig-zag poset} is defined to be a colored zig-zag poset $\zigipi$ with an arbitrary number of bars placed in each of the $n+1$ spaces to the left of $0_1$ such that between any two bars the elements of the poset (not necessarily their labels) are increasing.
This can be thought of as requiring at least one bar in each ``descent'' of the colored zig-zag poset.
Note that $0_1$ will always be to the right of all the bars.
See Figure~\ref{fig: barred col zig poset} for an example of a barred $\zigipi$ poset.
\begin{figure}[htbp]
\[\xymatrix @!R @!C @C=5pt{ \ar@{-}[dddd] &  & \ar@{-}[dddd] & & &  &  \ar@{-}[dddd] &  \ar@{-}[dddd] & &  & \\   &  &  & & &  3_1 \ar@{-}[ddrrr] &  &   & &  & \\  &  &  & 1_2 \ar@{-}[urr] &  & & &  & & & 0_1 \\  & 2_0 \ar@{-}[urr]  &  &  & & &  &  & 4_1 \ar@{-}[urr]& &  \\  &&&&&&&&&& }\] 
\caption{\; A barred $\zigipi$ poset with $I = \{3\}$ and $\pi = 2_0 1_2 3_1 4_1$.}
\label{fig: barred col zig poset}
\end{figure}

If we begin with a colored zig-zag poset $\zigipi$, then we must first place a bar in space $i$ for each $i \in I$.
From there we are free to place any number of bars in any of the $n+1$ spaces.
Define $\Omega_{\zigipi}(j,k)$ to be the number of ordered pairs $(f,P)$ where $P$ is a barred $\zigipi$ poset with $k$ bars and $f$ is a colored $\zigipi$-partition with parts in $\rversion{[0,j]}$.
Since the placement of bars is independent of the choice of colored $P$-partition $f$, we can compute $\Omega_{\zigipi}(j,k)$ by considering the two objects separately.
If we begin with the colored zig-zag poset $\zigipi$, then we first place one bar in space $i$ for each $i\in I$.
Next there are
\[
\mchoose{n+1}{k-|I|} = \binom{k+n - |I|}{n}
\]
ways to place the remaining $k-|I|$ bars in the $n+1$ spaces, and hence there are $\binom{k+n-|I|}{n}$ barred $\zigipi$ posets with $k$ bars.
Thus, using equation~\eqref{eq: colored order poly}, we have
\[
\Omega_{\zigipi}(j,k) = \sum_{\sigma \in \collin(\zigipi)} \Omega_{\sigma}(j) \binom{k+n-\des(\sigma^{-1}\pi)}{n} = \sum_{\sigma \in \collin(\zigipi)} \Omega_{\sigma}(j) \Omega_{\sigma^{-1}\pi}(k).
\]
If we set $\tau = \sigma^{-1}\pi$ and sum over all $I \subseteq [n]$, then we have
\begin{equation}\label{eq: col barred equiv}
\sum_{I \subseteq [n]} \Omega_{\zigipi}(j,k) = \sum_{\sigma \tau = \pi} \Omega_{\sigma}(j) \Omega_{\tau}(k).
\end{equation}

Next we define a \emph{barred colored chain poset} to be a colored chain poset $\chainipi$ with at least one bar in space $i$ for each $i \in I$ and with an arbitrary number of bars placed on the left end.
No bars are allowed in space $i$ for $i\in [n] \setminus I$.
Thus we place at least one bar between each chain of the colored chain poset and allow for bars on the left end.
Figure~\ref{fig: barred col chain poset} gives an example of a barred $\chainipi$ poset.
\begin{figure}[htbp]
\[\xymatrix @!R @!C @R=23pt @C=10pt{  \ar@{-}[ddd] &  & \ar@{-}[ddd]  & &  &  \ar@{-}[ddd] &  \ar@{-}[ddd] & &  \\   &  & & & 3_1   &  & &  &  0_1 \\  & 2_0  &  & 1_2\ar@{-}[ur] & & & & 4_1 \ar@{-}[ur] & \\  &   &  &  & & & & &  }\] 
\caption{\; A barred $\chainipi$ poset with $I = \{1, 3\}$ and $\pi = 2_0 1_2 3_1 4_1$.}
\label{fig: barred col chain poset}
\end{figure}
Define $\Omega_{\chainipi} (j,k)$ to be the number of ordered pairs $(f,P)$ where $P$ is a barred $\chainipi$ poset with $k$ bars and $f$ is a colored $\chainipi$-partition with parts in $\rversion{[0,j]}$.
The following lemma allows us to compare colored $P$-partitions for barred colored zig-zag posets and barred colored chain posets.

\begin{lem}\label{Barred Col Poset Extensions}
For every $\pi \in \colpermrn$,
\[
\sum_{I \subseteq [n]}  \Omega_{\zigipi}(j,k) = \sum_{I \subseteq [n]} \Omega_{\chainipi}(j,k).
\]
\end{lem}

\begin{proof}
We show that for every $\sigma \in \colpermrn$ there is a bijection between barred colored zig-zag posets with $k$ bars such that $\sigma$ is a linear extension of the underlying colored zig-zag poset and barred colored chain posets with $k$ bars such that $\sigma$ is a linear extension of the underlying colored chain poset.
This bijection is simply the map that sends each barred colored zig-zag poset to the barred colored chain poset obtained by removing the relation between $\pi(i)$ and $\pi(i+1)$ in $\zigipi$ for every space $i$ containing at least one bar.
\end{proof}

The bijection in Lemma~\ref{Barred Col Poset Extensions} maps Figure~\ref{fig: barred col zig poset} to Figure~\ref{fig: barred col chain poset}.
Now that all the pieces are in place, we are ready to prove the main result of this paper.

\begin{proof}[Proof of Theorem~\ref{Colored Descent Alg}]
First we see that
\begin{align*}
\frac{\dsum_{\sigma\tau = \pi} s^{\des(\sigma)} t^{\des(\tau)}}{(1-s)^{n+1}(1-t)^{n+1}} &=  \sum_{j,k \geq 0} \sum_{\sigma \tau = \pi} \binom{j+n}{n}\binom{k+n}{n}s^{j+\des(\sigma)} t^{k+\des(\tau)}\\
&= \sum_{j,k \geq 0} \sum_{\sigma \tau = \pi} \binom{j+n-\des(\sigma)}{n}\binom{k+n-\des(\tau)}{n}s^j t^k \\ 
&= \sum_{j,k \geq 0} \sum_{I \subseteq [n]} \Omega_{\zigipi}(j,k) s^j t^k,
\end{align*}
where the last equality follows from equations \eqref{eq: colored order poly} and \eqref{eq: col barred equiv}.
By Lemma~\ref{Barred Col Poset Extensions}, we can switch from colored zig-zag posets to colored chain posets, and we have
\[
\frac{\dsum_{\sigma\tau = \pi} s^{\des(\sigma)} t^{\des(\tau)}}{(1-s)^{n+1}(1-t)^{n+1}} = \sum_{j,k \geq 0} \sum_{I \subseteq [n]} \Omega_{\chainipi}(j,k) s^j t^k.
\]
The only remaining step is to prove that
\[
\sum_{I \subseteq [n]} \Omega_{\chainipi}(j,k) =  \binom{rjk+j+k+n-\des(\pi)}{n}.
\]

Fix a barred $\chainipi$ poset with $k$ bars and use the bars to define compartments numbered $0,\ldots,k$ from left to right.
Define $\pi_i$ to be the (possibly empty) subword of $\pi$ in compartment $i$ and denote the length of $\pi_i$ by $L_i$.
Then
\[
\Omega_{\chainipi}(j) = \Omega_{\pi_k}(j)\prod_{i=0}^{k-1} \Omega_{P(\pi_i)}(j) .
\]
By Theorem \ref{thm: P(pi) order poly}, we can assume that $\pi$ is monochromatic.
For $i=0,\ldots,k-1$, we let $\Omega_{P(\pi_i)}(j)$ count solutions to the inequalities
\[
i(rj+1) \leq s_{i_1} \leq  \cdots \leq s_{i_{L_i}} \leq i(rj+1) + rj,
\]
with $s_{i_l} < s_{i_{l+1}}$ if $l \in \intDes(\pi_i)$.
Next we let $\Omega_{\pi_k}(j)$ count solutions to the inequalities
\[
k(rj+1) \leq s_{k_1} \leq \cdots \leq s_{k_{L_k}} \leq rjk+j+k,
\]
with $s_{k_l} < s_{k_{l+1}}$ if $l \in \intDes(\pi_k)$ and $s_{k_{L_k}} < rjk+j+k$ if $n \in \Des(\pi)$.
By concatenating these inequalities, we see that if we sum over all $I \subseteq [n]$ and all barred $\chainipi$ posets with $k$ bars, then $\sum_{I \subseteq [n]} \Omega_{\chainipi}(j,k)$ is equal to the number of solutions to the inequalities
\[
0 \leq s_1 \leq \cdots \leq s_n \leq rjk+j+k,
\]
with $s_i < s_{i+1}$ if $i \in \intDes(\pi)$ and with $s_n < rjk+j+k$ if $n \in \Des(\pi)$.
Hence we conclude that
\[
\sum_{I \subseteq [n]} \Omega_{\chainipi}(j,k)  = \mchoose{rjk+j+k +1-\des(\pi)}{n} = \binom{rjk+j+k +n-\des(\pi)}{n}.
\]
\end{proof}

Figure~\ref{fig: every barred colored chain poset} provides a visual representation of the final step in the proof of Theorem~\ref{Colored Descent Alg}.
It depicts every barred $C(I,2_11_3)$ poset with $2$ bars and $I \subseteq \{1,2\}$.
Thus we can identify each barred $\chainipi$ poset with a barred colored permutation with underlying colored permutation $\pi$ such that $0_1$ is in the rightmost compartment.
\begin{figure}[htbp]
\begin{minipage}{0.3\textwidth}
\[\xymatrix @R=11pt @C=1pt {
&& \ar@{-}[ddd] & \ar@{-}[ddd] & \\
& 1_3 &&& \\
2_1\ar@{-}[ur] &&&& 0_1  \\
&&&&
}\]
\end{minipage}
\hfill
\begin{minipage}{0.3\textwidth}
\[\xymatrix @R=24pt @C=1pt {
& \ar@{-}[dd] && \ar@{-}[dd] & \\
2_1 && 1_3 && 0_1 \\
&&&&
}\]
\end{minipage}
\hfill
\begin{minipage}{0.3\textwidth}
\[\xymatrix @R=11pt @C=1pt {
& \ar@{-}[ddd] & \ar@{-}[ddd] && \\
&&&& 0_1 \\
2_1 &&& 1_3\ar@{-}[ur] &   \\
&&&&
}\]
\end{minipage}

\begin{minipage}{0.3\textwidth}
\[\xymatrix @R=10pt @C=1pt {
\ar@{-}[ddd] &&& \ar@{-}[ddd] & \\
&& 1_3 && \\
& 2_1\ar@{-}[ur] &&& 0_1 \\
&&&&
}\] 
\end{minipage}
\hfill
\begin{minipage}{0.3\textwidth}
\[\xymatrix @R=10pt @C=1pt {
\ar@{-}[ddd] && \ar@{-}[ddd] && \\
&&&& 0_1  \\
& 2_1 && 1_3\ar@{-}[ur] && \\
&&&&
}\] 
\end{minipage}
\hfill
\begin{minipage}{0.3\textwidth}
\[\xymatrix @R=4pt @C=1pt {
\ar@{-}[dddd]  & \ar@{-}[dddd] &&& \\
&&&& 0_1 \\
&&& 1_3\ar@{-}[ur] & \\
&& 2_1\ar@{-}[ur] && \\
&&&&
}\] 
\end{minipage}
\caption{All six barred $C(I,2_11_3)$ posets with $2$ bars and $I \subseteq \{1,2\}$.}
\label{fig: every barred colored chain poset}
\end{figure}
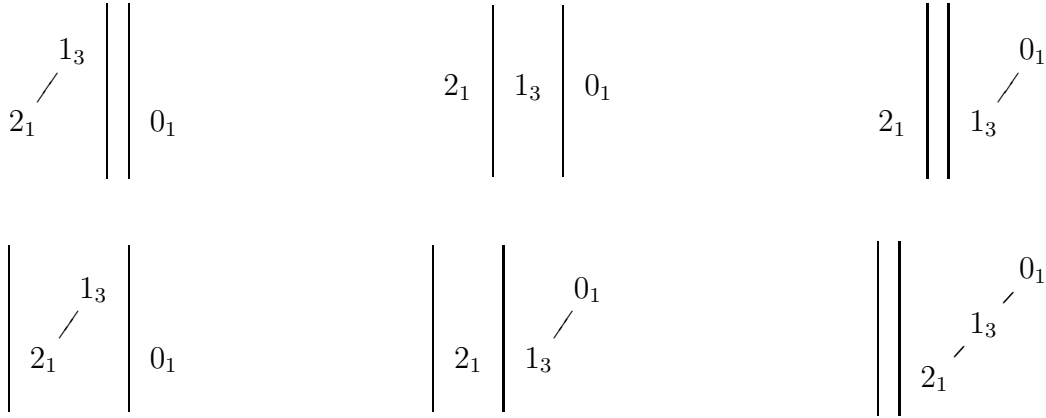

\newpage

\section{Colored Eulerian Idempotents}
\label{sec: colored eulerian idempotents}

Finally, we will explore a change of basis for the colored Eulerian descent algebra.
There exists a basis of pairwise orthogonal idempotent elements which generalize the familiar Eulerian idempotents originally found in \cite{MielnikPlebanski1970} and the type~$B$ Eulerian idempotents of \cite{BergeronBergeron1992}.
Our proof technique mirrors that of other authors (see \cite{ChowThesis2001, Petersen2005}).
Define the \emph{colored structure polynomial} $\phi(x)$ by
\[
\phi(x) = \sum_{\pi \in \colpermrn} \binom{x+n-\des(\pi)}{n} \pi.
\]
If we expand both sides of equation~\eqref{eq: col des algebra} and compare the coefficients of $s^j t^k$, we see that
\[
\binom{rjk+j+k+n-\des(\pi)}{n} = \sum_{\sigma\tau = \pi}  \binom{j+n -\des(\sigma)}{n} \binom{k+n - \des(\tau)}{n}.
\]
This implies that $\phi(j)\phi(k) = \phi(rjk + j +k)$ for all $j,k \geq 0$.
Thus we have the following theorem.

\begin{thm}\label{Orthogonal Idempotents}
As polynomials in $x$ and $y$ with coefficients in the group algebra of $\colpermrn$,
\[
\phi(x)\phi(y) = \phi(rxy+x+y).
\]
\end{thm}

If we substitute $x \leftarrow (x-1)/r$ and $y \leftarrow (y-1)/r$ into the previous theorem, we see that $\phi((x-1)/r)\phi((y-1)/r) = \phi((xy-1)/r)$.
Thus if we expand $\phi((x-1)/r)$, we have
\[
\phi((x-1)/r) = \sum_{\pi \in \colpermrn} \binom{\frac{x-1}{r}+n-\des(\pi)}{n} \pi = \sum_{i=0}^n c_i x^i,
\]
and the $c_i$ are orthogonal idempotents which span the colored Eulerian descent algebra.
To see why, note that $\phi((x-1)/r)\phi((y-1)/r) = \phi((xy-1)/r)$ implies that
\[
\sum_{i,j=0}^n c_ic_j x^iy^j = \sum_{i=0}^n c_i (xy)^i.
\]
Thus $c_ic_j = 0$ unless $i=j$, in which case $c_i^2 = c_i$.
The $c_i$ reduce to the familiar Eulerian idempotents $e_i$ when $r=1$.
When $r=2$, our orthogonal idempotents are equivalent to the type~$B$ orthogonal idempotents originally found in \cite{BergeronBergeron1992}.

\begin{ex}
Let $r=5$ and $n=3$.
If $C_i$ is the sum of all permutations in $G_{5,3}$ with $i$~descents, then we end up with the following orthogonal idempotents:
\begin{align*}
c_0 &= \frac{1}{750} (504 C_0 - 36 C_1 + 24 C_2 - 66 C_3) \\
c_1 &= \frac{1}{750} (218 C_0 + 23 C_1 - 22 C_2 + 83 C_3) \\
c_2 &= \frac{1}{750} (27 C_0 + 12 C_1 - 3 C_2 - 18 C_3) \\
c_3 &= \frac{1}{750} (C_0 + C_1 + C_2 + C_3) .
\end{align*}
\end{ex}

We end with two negative results to natural questions.
First, if we consider all possible descent set definitions produced by setting $\pi(0) = 0_a$ and $\pi(n+1) = 0_b$, then the definition we use induces the only set partition that gives rise to an algebra.
This can be easily verified for $G_{2,2}$.

Finally, it is natural to ask whether there is a colored descent algebra induced by the descent set.
The descent set statistic induces a set partition of $\colpermrn$ that would fall between the set partition induced by descent number and the set partition corresponding to the Mantaci-Reutenauer algebra, and thus the descent set statistic would induce an intermediate algebra.
Interestingly, we can see that the descent set statistic already fails to induce an algebra in the case of $G_{2,2}$.


\section*{Acknowledgements}

I would like to thank Ira Gessel for all his guidance and Rachel Bayless for all her helpful suggestions.

\bibliographystyle{amsplain}
\bibliography{CEDAbibliography}

\end{document}